\documentclass{amsart}

\usepackage{amsmath, amsthm, amssymb}
\usepackage{tikz}
\usetikzlibrary{decorations.pathreplacing}
\usepackage{graphicx}
\usepackage{array}
\usepackage{xcolor}

\newcommand{\weight}[2]{W(#1,#2)}

\theoremstyle{definition} \newtheorem{definition}{Definition}[section]
\theoremstyle{plain} \newtheorem{theorem}[definition]{Theorem}
\theoremstyle{plain} 
\theoremstyle{plain} \newtheorem{proposition}[definition]{Proposition}
\theoremstyle{plain} 
\theoremstyle{plain} 
\theoremstyle{definition} \newtheorem{example}[definition]{Example}
\theoremstyle{remark} \newtheorem{remark}[definition]{Remark}
\theoremstyle{remark} \newtheorem{question}[definition]{Question}

\usepackage{xspace}
\makeatletter
\DeclareRobustCommand\onedot{\futurelet\@let@token\@onedot}
\def\@onedot{\ifx\@let@token.\else.\null\fi\xspace}
 \def\Eg{{E.g}\onedot}
\def\ie{{i.e}\onedot}

\def\etal{{et al}\onedot}
\makeatother

\begin{document}

\title{Games and Complexes II: Weight Games and Kruskal-Katona Type Bounds}
\author{Sara Faridi \and Svenja Huntemann \and Richard J.~Nowakowski}

\address{Dalhousie University\\
Halifax, Nova Scotia\\
B3H 3J5, Canada}
\email{faridi@mathstat.dal.ca, svenjah@mathstat.dal.ca, rjn@mathstat.dal.ca}


\thanks{The first and third authors' research was supported in part by the Natural Sciences and Engineering Research Council of Canada.}

\begin{abstract}
A strong placement game $G$ played on a board $B$ is equivalent to a simplicial complex $\Delta_{G,B}$. We look at weight games, a subclass of strong placement games, and introduce upper bounds on the number of positions with $i$ pieces in $G$, or equivalently the number of faces with $i$ vertices in $\Delta_{G,B}$, which are reminiscent of the Kruskal-Katona bounds.
\end{abstract}

\maketitle

\section{Introduction}
Our goal in this paper is to study complexes of placement games (Definition \ref{def:placement}). In \cite{GameCompI} we demonstrated that to a placement game $G$ played on a board $B$ one can associate a simplicial complex $\Delta_{G,B}$ where $G$ can be considered as a game played on $\Delta_{G,B}$.

The main questions that we address in this paper is: What complexes can be legal complexes of a placement game?

We give partial answers to this question in specific cases: when the board is a path, a cycle, or a complete graph (also see \cite{MSc}).

We begin by introducing some of the concepts needed. A complete introduction is given in \cite{GameCompI}.

\begin{definition}\label{def:placement}
A \textit{strong placement game} is a combinatorial game played on a graph which satisfies the following:
\begin{itemize}
\item[(i)] The starting position is the empty board.
\item[(ii)] Players place pieces on empty spaces of the board according to the rules.
\item[(iii)] Pieces are not moved or removed once placed.
\item[(iv)] The rules are such that if it is possible to reach a position through a sequence of legal moves, then any sequence of moves leading to this position consists only of legal moves.
\end{itemize}
The \textsc{Trivial} placement game on a board is the placement game that has no additional rules.
\end{definition}

Throughout this paper `placement game' refers to a strong placement game. Since placement games are played on a graph, we use the terms board and graph, and space and vertex interchangeably.

A \textit{basic position} is a board with only one piece placed. Any position, whether legal or illegal, in a placement game can be decomposed into basic positions.

\begin{definition}
A \textit{simplicial complex} $\Delta$ on a finite vertex set $V$ is a set of subsets (called \textit{faces}) of $V$ with the conditions that if $A\in \Delta$ and $B\subseteq A$, then $B\in \Delta$. The \textit{$f$-vector} $(f_0, f_1, \ldots, f_k)$ of a simplicial complex $\Delta$ enumerates the number of faces $f_i$ with $i$ vertices. Note that if $\Delta\neq\emptyset$, then $f_0=1$.
\end{definition}

The \textit{legal complex} \cite{GameCompI}, denoted by $\Delta_{G,B}$, is the simplicial complex whose faces correspond to the legal positions of the placement game $G$ played on the board $B$.

\begin{question}\label{Q}
Is every simplicial complex the legal complex of a placement game?
\end{question}

In respect to this question, we are interested in the possible $f$-vectors of legal complexes, thus we will consider the following:

The number of positions in $G$ on $B$ with $i$ pieces played, or equivalently the number of faces with $i$ vertices in the legal complex $\Delta_{G,B}$, is denoted by $f_i(G,B)$, or shortened to $f_i$ if the game and board are clear. In this work, we will be considering upper bounds on $f_i(G,B)$. Specifically, we will be considering Kruskal-Katona type bounds for weight games played on a path, on a cycle, or on a complete graph.

\section{The Kruskal-Katona Theorem}

Kruskal \cite{Kr63} and Katona \cite{Ka68} proved that for each pair of non-negative integers $f$ and $i$, $f$ can be written in the form
\[f=\binom{n_i}{i}+\binom{n_{i-1}}{i-1}+\ldots+\binom{n_{i-s}}{i-s}\] where $n_i>n_{i-1}>\ldots>n_{i-s}\geq i-s\geq 1$ are unique. This sum is called the \textit{$i$-canonical representation of $f$}.

We can then define the \textit{$j$th pseudopower of $f$}
\[f^{(j)}_i=\binom{n_i}{j}+\binom{n_{i-1}}{j-1}+\ldots+\binom{n_{i-s}}{j-s}\]
for $j\ge 1$.

The Kruskal-Katona theorem gives necessary and sufficient conditions for a vector $(f_0, f_1,\ldots, f_k)$ with entries from the non-negative integers to be the $f$-vector of a simplicial complex. The following is the version proven by Kruskal:

\begin{theorem}[Kruskal \cite{Kr63}]
	For the sequence of non-negative integers $(f_0, f_1,\ldots, f_k)$ the following are equivalent:
	\begin{itemize}
		\item[\emph{(i)}] $(f_0, f_1,\ldots, f_k)$ is the $f$-vector of a non-empty simplicial complex;
		\item[\emph{(ii)}] $f_0=1$ and $f_{j}\le f_{i}^{(j)}$ for all $1\le i\le j$;
		\item[\emph{(iii)}] $f_0=1$ and $f_{j}\ge f_{i}^{(j)}$ for all $1\le j\le i$.
	\end{itemize}
\end{theorem}

To show that (ii) holds, it is sufficient to show that $f_0=1$ and $f_{i+1}\le f_i^{(i+1)}$ for all $i\ge 1$ since all other cases follow. Similarly, to show (iii), showing $f_0=1$ and $f_j\ge f_{j+1}^{(j)}$ for all $j\ge 1$ is sufficient. The Kruskal-Katona theorem is usually stated in terms of either one of these.

If the answer to Question \ref{Q} is ``no'', then not every vector that is an $f$-vector of a simplicial complex is also an $f$-vector of a legal complex. Thus for the remainder, after introducing weight games, we will give improved upper bounds on the entries of an $f$-vector of a legal complex.

\section{Games with Weight}\label{sec:weight}

In the remainder, we will consider playing pieces of larger size. Specifically, we call the number of connected vertices a piece covers the \textit{weight} of this piece.

Many placement games have pieces of weight greater than 1. For example, in \textsc{Domineering} \cite{BCG04} and \textsc{Crosscram} \cite{Ga74} Left and Right both play dominoes as their pieces, and so their pieces are of weight 2. Also, as we will mention in Remark \ref{remark:octal}, partizan octal games are equivalent to placement games with weight on a path.

\begin{example}
Consider the board given in Figure \ref{fig:weightex}. A piece that has weight 4 could for example be played on the vertex set $\{1,2,3,4\}$, but not on the vertex set $\{1,3,5,6\}$ since these vertices are not connected.
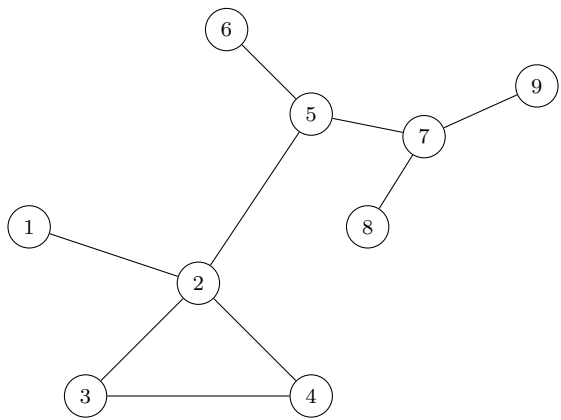
\begin{figure}[!ht]
\begin{center}
	\begin{tikzpicture}[scale=1.5]
		\node[shape=circle, draw] (1) at (-0.5,0) {\footnotesize 1};
		\node[shape=circle, draw] (2) at (1, -0.5) {\footnotesize 2};
		\node[shape=circle, draw] (3) at (0,-1.5) {\footnotesize 3};
		\node[shape=circle, draw] (4) at (2, -1.5) {\footnotesize 4};
		\node[shape=circle, draw] (5) at (2,1) {\footnotesize 5};
		\node[shape=circle, draw] (6) at (1.25, 1.75) {\footnotesize 6};
		\node[shape=circle, draw] (7) at (3, 0.8) {\footnotesize 7};
		\node[shape=circle, draw] (8) at (2.5,0) {\footnotesize 8};
		\node[shape=circle, draw] (9) at (4, 1.25) {\footnotesize 9};
		\draw (1) to (2) to (3) to (4) to (2) to (5) to (7) to (8);
		\draw (5) to (6);
		\draw (7) to (9);
	\end{tikzpicture}
\end{center}
\caption{An Example Board}
\label{fig:weightex}
\end{figure}
\end{example}

We usually assume that every piece of Left has the same weight $a$, and every piece of Right has the same weight $b$. 

\begin{definition}\label{def:weightgame}
A placement game in which the players play pieces of fixed weights is called a \textit{game with weights}. If the game has no rules besides pieces having to be placed on connected sets of empty vertices, we call it a \textit{weight game}. A $2$-player weight game will be denoted by $\weight{a}{b}$ where $a$ is the weight that Left plays, while $b$ is the weight that Right plays.
\end{definition}

Essentially, the weight game is the \textsc{Trivial} placement game with weights. 

In \cite{MSc}, it is shown that the game $\weight{a}{a}$ played on a path or a cycle is equivalent to another placement game in which both Left and Right play pieces of weight $1$. This is not necessarily true though if we force every basic position to be legal, as the following discussion shows.

Consider a placement game $G$ in which both Left and Right play pieces of weight $1$ and every basic position is legal. Since the basic positions in this case consists of Left or Right occupying a single vertex, we have $n$ Left and Right basic positions each, where $n$ is the number of vertices of the board. Thus we have that the number of legal positions with one piece is the number of basic positions, namely $f_1=2n$.

This also implies that a weight game $\weight{a}{b}$ where $f_1$ is odd is not equivalent to a placement game where both Left and Right play pieces of weight $1$ and basic positions are legal. Weight games with $f_1$ odd indeed exist, as seen in the following example.
\begin{example}
Consider $\weight{1}{2}$ played on $P_2$. The basic positions are
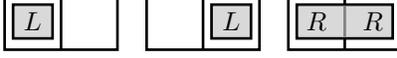
\begin{center}
	\begin{tikzpicture}[scale=0.75]
		\draw[line width=1pt] (0,0)--(2,0)--(2,-1)--(0,-1)--cycle;
		\draw[line width=1pt] (1,0)--(1,-1);
		\filldraw[fill=gray!30, line width=1pt] (0.15,-0.2)--(0.85,-0.2)--(0.85,-0.8)--(0.15,-0.8)-- cycle;
		\draw (0.5,-0.5) node {$L$};
	\end{tikzpicture}\quad
	\begin{tikzpicture}[scale=0.75]
		\draw[line width=1pt] (0,0)--(2,0)--(2,-1)--(0,-1)--cycle;
		\draw[line width=1pt] (1,0)--(1,-1);
		\filldraw[fill=gray!30, line width=1pt] (1.15,-0.2)--(1.85,-0.2)--(1.85,-0.8)--(1.15,-0.8)-- cycle;
		\draw (1.5,-0.5) node {$L$};
	\end{tikzpicture}\quad
	\begin{tikzpicture}[scale=0.75]
		\draw[line width=1pt] (0,0) -- (2,0) -- (2, -1) -- (0, -1) -- cycle;
		\draw[line width=1pt] (1,0) -- (1, -1);
		
		\filldraw[fill=gray!30, line width=1pt] (0.15,-0.2)--(1.85,-0.2)--(1.85,-0.8)--(0.15,-0.8)-- cycle;		
		\draw (0.5,-0.5) node {$R$};
		\draw (1.5,-0.5) node {$R$};		
		\draw[draw=gray, line width=1pt] (1,-0.23)--(1,-0.77);		
	\end{tikzpicture}
	\end{center}
and thus if all basic positions are legal, then $f_1=3$.
\end{example}

Suppose the weight of the Left pieces is $a$ and the weight of the Right pieces $b$ and without loss of generality $a\le b$, then Left would be able to place at most $\lfloor n/a\rfloor$ pieces on a board of $n$ vertices. If we place a mix of Left and Right pieces or just Right pieces, the number of pieces we are able to place will be equal or less. Thus if the $f$-vector of the legal complex is $(f_0, f_1, \ldots, f_k)$, then \[k\le\max\{\left\lfloor n/a\right\rfloor, \left\lfloor n/b\right\rfloor\}.\]

\begin{proposition}\label{thm:weight1bound}
	For legal complexes corresponding to games on any board of $n$ vertices with pieces of weight 1, we have
	\begin{equation*}
		f_i\leq \binom{n}{i}2^i
	\end{equation*}
	for $i\ge 0$.
\end{proposition}
\begin{proof}
	We will consider the number of positions with $i$ pieces of weight 1 in the placement game that has no additional rules, \ie the \textsc{Trivial} placement game. As we add rules to this game to get other placement games with pieces of weight 1, the number of positions decreases, thus the number of such positions in \textsc{Trivial} gives the maximum. In \textsc{Trivial}, there are $\binom{n}{i}$ ways to choose $i$ spaces to place pieces, for each there are 2 choices: either a Left piece, or a Right piece. Our claim now follows.
\end{proof}

We will now look at how playing pieces of specified weight on different classes of boards influences the $f$-vector of the corresponding legal complex. The classes of boards we specifically look at are paths, cycles, and complete graphs.

Note that the $f$-vector of a weight game gives an upper bound on the $f$-vector of a game with the same weights. Thus the formulae for the weight games in the following sections give bounds for games with weight.

In \cite{MSc}, these formulae are also generalized to $t$-player weight games.

\section{Playing on the Path $P_n$}
In this section, we study placement games played on the path $P_n$, $n\ge 1$, in which Left plays pieces of weight $a$ and Right pieces of weight $b$.
\begin{proposition}\label{thm:path1}
If a simplicial complex is the legal complex of a weight game $\weight{a}{b}$ played on $P_n$ then
	\begin{equation}
    f_1=
    \begin{cases}
    	0 & \text{if } a,b>n,\\
    	n-a+1 & \text{if } a\le n \text{ and } b>n,\\
    	n-b+1 & \text{if } a> n \text{ and } b\le n,\\
    	2n-a-b+2 & \text{if } a,b\le n.
    \end{cases}
	\end{equation}
\end{proposition}
\begin{proof}
    We are measuring the number of legal basic positions. If $a,b>n$, then neither Left nor Right can place a piece, thus $f_1=0$. If $n\ge a$, then placing one piece of weight $a$ on a strip of length $n$ is equivalent to placing one piece of weight $1$ (think of the left-most end of the piece) on a strip of length $n-(a-1)=n-a+1$, so the second and third case follow. Similarly, for the final case
    \begin{align*}
    	f_1&= (n-a+1)+(n-b+1)\\
    	&=2n-a-b+2.\qedhere
	\end{align*}
\end{proof}

\begin{proposition}\label{thm:path2}
In a weight game $\weight{a}{b}$ played on $P_n$, the number of positions with one Left and one Right piece is
\begin{equation*}
    N_{LR}=
    \begin{cases}
    	0 & \text{if } a+b>n,\\
    	2\binom{n-a-b+2}{2} & \text{if } a+b\le n.
    \end{cases}
\end{equation*}

The number of positions with two Left pieces or two Right pieces, respectively, is
\begin{equation*}
    N_{LL}=
    \begin{cases}
    	0 & \text{if } 2a>n,\\
    	\binom{n-2a+2}{2} & \text{if } 2a\le n;
    \end{cases}\qquad
     N_{RR}=
    \begin{cases}
    	0 & \text{if } 2b>n,\\
    	\binom{n-2b+2}{2} & \text{if } 2b\le n.
    \end{cases}
\end{equation*}

For the legal complex of such a game we have
\begin{equation}\label{eq:path2final}
	f_2= N_{LR}+N_{LL}+N_{RR}.
\end{equation}
\end{proposition}
\begin{proof}
To find $N_{LR}$ when $n\ge a+b$, we only consider the case in which the Left piece is the left-most piece. The other case is symmetric. We will first place the Left piece in position $i$. To be able to fit a Right piece to the right of this, we have $1\le i\le n-a-b+1$.
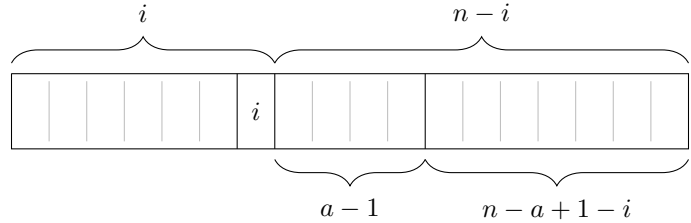
\begin{figure}[!ht]
	\begin{center}
	\begin{tikzpicture}
		\draw (0,1)--(9,1)--(9,0)--(0,0)--(0,1);
		\draw (3,1)--(3,0);
		\draw (3.5,1)--(3.5,0);
		\draw (5.5,1)--(5.5,0);
		\node at (3.25,0.5) {$i$};
		\foreach \x in {0.5,1,1.5,2,2.5,4,4.5,5,6,6.5,7,7.5,8,8.5}{
			\draw[gray!50] (\x,0.9)--(\x,0.1);}
		\draw [decorate,decoration={brace,amplitude=10pt,raise=4pt},yshift=0pt]
(0,1) -- (3.5,1) node [black,midway,yshift=0.8cm] {$i$};
		\draw [decorate,decoration={brace,amplitude=10pt,mirror,raise=4pt},yshift=0pt]
(3.5,0) -- (5.5,0) node [black,midway,yshift=-0.8cm] {$a-1$};
		\draw [decorate,decoration={brace,amplitude=10pt,raise=4pt},yshift=0pt]
(3.5,1) -- (9,1) node [black,midway,yshift=0.8cm] {$n-i$};
		\draw [decorate,decoration={brace,amplitude=10pt,mirror,raise=4pt},yshift=0pt]
(5.5,0) -- (9,0) node [black,midway,yshift=-0.8cm] {$n-a+1-i$};
	\end{tikzpicture}
	\end{center}
	\caption{Proof to Proposition \ref{thm:path2}: Placing a Piece of Weight $a$ on a Path}
	\label{fig:pathbreakup}
\end{figure}
The strip to the right then has length $n-(i+a-1)=n-a+1-i$ (see Figure \ref{fig:pathbreakup}). Thus we have $n-a+1-i-(b-1)=n-a-b+2-i$ choices to place the Right piece (see Proposition \ref{thm:path1}). Thus the number of position with Left on the left and Right on the right is
\begin{align*}
	&\sum_{i=1}^{n-a-b+1} (n-a-b+2-i)\\
	=&(n-a-b+1)(n-a-b+2)-\sum_{i=1}^{n-a-b+1} i\\
	=&(n-a-b+1)(n-a-b+2)-\frac{(n-a-b+1)(n-a-b+2)}{2}\\
	=&\frac{(n-a-b+1)(n-a-b+2)}{2}\\
	=&\binom{n-a-b+2}{2}.
\end{align*}
Then $N_{LR}=2\binom{n-a-b+2}{2}$.

Similarly, the number of positions with Left on the left and right for $n\ge 2a$ and Right on the left and right for $n\ge 2b$ respectively, then are
\[ N_{LL}=\binom{n-2a+2}{2}\qquad N_{RR}=\binom{n-2b+2}{2}.\]

Since these are the only three possibilities for pairs of pieces, Equation \ref{eq:path2final} follows immediately.
\end{proof}

It is easy to see that if $a=b=1$, then the previous two bounds are
\begin{align*}
	f_1&=2n;\\
	f_2&=4\binom{n}{2}.
\end{align*}
These are the bounds given in Proposition \ref{thm:weight1bound}.
      	
\begin{example}
Consider $\weight{2}{3}$ on the path $P_5$. Let $x_i$ represent a Left piece occupying the spaces $i$ and $i+1$, and similarly for $y_i$. For example, the position in Figure \ref{fig:pathpos} is represented by $x_1y_3$.

\begin{figure}[!ht]
	\begin{center}
	\begin{tikzpicture}[scale=1]
		\draw[line width=1pt] (0,0) -- (5,0) -- (5, -1) -- (0, -1) -- cycle;
		\draw[line width=1pt] (1,0) -- (1, -1);
		\draw[line width=1pt] (2,0) -- (2, -1);
		\draw[line width=1pt] (3,0) -- (3, -1);
		\draw[line width=1pt] (4,0) -- (4, -1);
		\filldraw[fill=gray!30, line width=1pt] (0.15,-0.2)--(1.85,-0.2)--(1.85,-0.8)--(0.15,-0.8)-- cycle;
		\filldraw[fill=gray!30, line width=1pt] (2.15,-0.2)--(4.85,-0.2)--(4.85,-0.8)--(2.15,-0.8)-- cycle;
		\draw (0.5,-0.5) node {$L$};
		\draw (1.5,-0.5) node {$L$};
		\draw (2.5,-0.5) node {$R$};
		\draw (3.5,-0.5) node {$R$};
		\draw (4.5,-0.5) node {$R$};
		\draw[draw=gray, line width=1pt] (1,-0.23)--(1,-0.77);
		\draw[draw=gray, line width=1pt] (3,-0.23)--(3,-0.77);
		\draw[draw=gray, line width=1pt] (4,-0.23)--(4,-0.77);
	\end{tikzpicture}
	\end{center}
	\caption{An Example Position for $\weight{2}{3}$ on $P_5$}
	\label{fig:pathpos}
\end{figure}
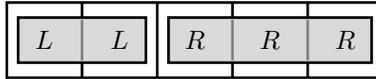      		

The corresponding simplicial complex is given in Figure \ref{fig:pathsc}.

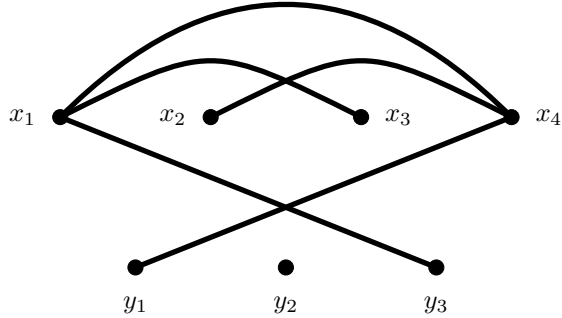
\begin{figure}[!ht]
	\begin{center}
	\begin{tikzpicture}[scale=1]
		\draw[line width=2pt] (0,0) .. controls (2,2) and (4,2) .. (6,0);
		\draw[line width=2pt] (0,0) .. controls (2,1) .. (4,0);
		\draw[line width=2pt] (2,0) .. controls (4,1) .. (6,0);
		\draw[line width=2pt] (0,0)--(5,-2);
		\draw[line width=2pt] (6,0)--(1,-2);
		\filldraw (0,0) circle (0.1cm);
		\filldraw (2,0) circle (0.1cm);
		\filldraw (4,0) circle (0.1cm);
		\filldraw (6,0) circle (0.1cm);
		\filldraw (1,-2) circle (0.1cm);
		\filldraw (3,-2) circle (0.1cm);
		\filldraw (5,-2) circle (0.1cm);
		\draw (-0.5,0) node {$x_1$};
		\draw (1.5,0) node {$x_2$};
		\draw (4.5,0) node {$x_3$};
		\draw (6.5,0) node {$x_4$};
		\draw (1,-2.5) node {$y_1$};
		\draw (3,-2.5) node {$y_2$};
		\draw (5,-2.5) node {$y_3$};
	\end{tikzpicture}
	\end{center}
	\caption{The Legal Complex $\Delta_{\weight{2}{3},P_5}$}
	\label{fig:pathsc}
\end{figure}

By Propositions \ref{thm:path1} and \ref{thm:path2} we have
\begin{align*}
	f_0&=1\\
	f_1&=2n-a-b+2=7,\\
	f_2&=\binom{n-2a+2}{2}+2\binom{n-a-b+2}{2}=5,
\end{align*}
and since $\max\{\left\lfloor n/a\right\rfloor, \left\lfloor n/b\right\rfloor\}=2$, we get the $f$-vector $(1, 7, 5)$, which can be verified from the simplicial complex.

To compare this with the Kruskal-Katona bound, we first need to find the $i$-canonical representations and calculate the $j$th pseudopowers.
\begin{center}
\begin{tabular}{ll}
	$f_1=\binom{7}{1}$ & $f_1^{(2)}=\binom{7}{2}=21$\\
	$f_2=\binom{3}{2}+\binom{2}{1}$ & $f_2^{(3)}=\binom{3}{3}+\binom{2}{2}=2$\\
	 & $f_2^{(1)}=\binom{3}{1}+\binom{2}{0}=4$
\end{tabular}
\end{center}
Then $f_2=5<f_1^{(2)}=21$, $f_3=0<f_2^{(3)}=2$, and $f_1=7>f_2^{(1)}=4$, showing that the formulae in Propositions \ref{thm:path1} and \ref{thm:path2} give, at least for this example, improved necessary conditions for a vector to be the $f$-vector of a legal complex of a placement game played on a path over the ones given in the Kruskal-Katona theorem.
\end{example}

We will now show that for fixed $a$ and $b$ and sufficiently large $n$, then the bound in Proposition \ref{thm:path2} on $f_2$ is better than the Kruskal-Katona bound. By the Kruskal-Katona theorem we have
\begin{align*}
	f_2\le f_1^{(2)}&=\binom{2n-a-b+2}{2}\\
		&=\frac{1}{2}\left[4n^2+n(6-4a-4b)+g(a,b)\right],
\end{align*}
where $g(a,b)$ is a function in $a$ and $b$, whereas Proposition \ref{thm:path2} gives
\begin{align*}
	f_2=&\binom{n-2a+1}{2}+\binom{n-2b+1}{2}+2\binom{n-a-b+1}{2}\\
	=&\frac{1}{2}\left[4n^2+2n(6-4a-4b)+h(a,b)\right],
\end{align*}
where $h(a,b)$ is a function in $a$ and $b$. Since $a,b\ge 1$, and thus $6-4a-4b<0$, we have $\frac{1}{2}\left[4n^2+2n(6-4a-4b)+g(a,b)\right]< \frac{1}{2}\left[4n^2+n(6-4a-4b)+h(a,b)\right]$ for sufficiently large $n$, showing that as $n$ grows larger our bound becomes increasingly better than the Kruskal-Katona bound.

\begin{remark}\label{remark:octal}
The game \textsc{O12} is the weight game $\weight{1}{2}$. It is mentioned by Brown \etal in \cite{GamePol} that this game played on a path is equivalent to the partizan Octal game where Left removes one piece and Right two, and both have the possibility to split the heap. It is easy to see that weight games played on a path are all equivalent to a specific partizan Octal game.
\end{remark}

\section{Playing on the Cycle $C_n$}
Consider Left playing pieces of weight $a$ and Right pieces of weight $b$ on a cycle of length $n\ge 3$. For this board, the `left' end of a piece is the end in counter-clockwise direction.
\begin{proposition}\label{thm:cycle1}
If a simplicial complex is the legal complex of $\weight{a}{b}$ played on $C_n$ then
\begin{equation}
	f_1= 
    \begin{cases}
    	0 & \text{if } a,b>n,\\
    	n & \text{if either } a\le n \text{ or } b\le n \text{ but not both},\\
    	2n & \text{if } a,b\le n.
    \end{cases}
\end{equation}
\end{proposition}
\begin{proof}
	The left end of a piece can be placed on any of the $n$ spaces if its weight is less than $n$, no matter if it is a Right or Left piece.
\end{proof}

\begin{proposition}\label{thm:cycle2}
If a simplicial complex is the legal complex of $\weight{a}{b}$ played on $C_n$ then
\begin{equation}
	f_2= N_{LL}+N_{LR}+N_{RR}
\end{equation}
where
\begin{equation*}
    N_{LL}=
    \begin{cases}
    	0 & \text{if } 2a>n,\\
    	\frac{n(n-2a+1)}{2} & \text{if } 2a\le n,
    \end{cases}\qquad
      N_{RR}=
    \begin{cases}
    	0 & \text{if } 2b>n,\\
    	\frac{n(n-2b+1)}{2} & \text{if } 2b\le n,
    \end{cases}
\end{equation*}
are the number of positions with two Left pieces, respectively two Right pieces, and
\begin{equation*}
    N_{LR}=
    \begin{cases}
    	0 & \text{if } a+b>n,\\
    	n(n-a-b+1) & \text{if } a+b\le n,
    \end{cases}
\end{equation*}
is the number of positions with one Left and one Right piece.
\end{proposition}
\begin{proof}
	We will first look at the number of positions with two Left pieces if $n\ge 2a$. There are $n$ choices for placing the first piece. Placing the second piece is equivalent to placing one piece on the path $P_{n-a}$, i.e. there are $(n-a)-(a-1)$ choices for placing the second piece. Due to symmetry, there are then $n(n-2a+1)/2$ positions of this form. Similarly, the number of positions with two Right pieces is $n(n-2b+1)/2$ if $n\ge 2b$.
		
	To count the number of positions with one Left and one Right piece when $n\ge a+b$, we first place the Left, then the Right piece. There are $n$ choices for placing the Left piece. Placing the Right piece is then equivalent to placing a piece of weight $b$ on the path $P_{n-a}$, i.e. there are $(n-a)-(b-1)$ choices for placing the second pieces. Thus, there are $n(n-a-b+1)$ positions of this form.
\end{proof}

If $a=b=1$, then the previous two bounds are
\begin{align*}
	f_1&=2n;\\
	f_2&=4\binom{n}{2},
\end{align*}
which are the bounds given in Proposition \ref{thm:weight1bound}.

\begin{example}
Consider $\weight{2}{3}$ on the cycle $C_5$. Let $x_i$ represent a Left piece whose left end is on space $i$, and similarly for $y_i$. \Eg the position in Figure \ref{fig:C5ex} is represented by $x_1y_3$.
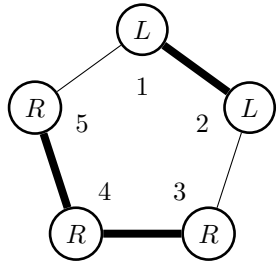
\begin{figure}[!ht]
	\begin{center}
	\begin{tikzpicture}[scale=0.5]
		\node[shape=circle, draw, line width=1pt] (1) at (0,3) {$L$};
		\node[shape=circle, draw, line width=1pt] (3) at (1.763,-2.427) {$R$};
		\node[shape=circle, draw, line width=1pt] (4) at (-1.763,-2.427) {$R$};
		\node[shape=circle, draw, line width=1pt] (2) at (2.853,0.927) {$L$};
		\node[shape=circle, draw, line width=1pt] (5) at (-2.853,0.927) {$R$};
		\draw[line width=3pt] (1)--(2);
		\draw (2)--(3);
		\draw[line width=3pt] (3)--(4);
		\draw[line width=3pt] (4)--(5);
		\draw (5)--(1);
		\node at (0, 1.5) {$1$};
		\node at (1.6,0.5) {$2$};
		\node at (1,-1.3) {$3$};
		\node at (-1, -1.3) {$4$};
		\node at (-1.6, 0.5) {$5$};
	\end{tikzpicture}
	\end{center}
	\caption{An Example Position for $\weight{2}{3}$ on $C_5$}
	\label{fig:C5ex}
\end{figure}

\noindent The corresponding legal complex is given in Figure \ref{fig:cyclesc}.

\begin{figure}[!h]
	\begin{center}
	\begin{tikzpicture}[scale=1]
		\draw[line width=2pt] (0,3) -- (0,2);
		\draw[line width=2pt] (2.853,0.927) -- (1.902,0.618);
		\draw[line width=2pt] (1.763,-2.427) -- (1.176,-1.618);
		\draw[line width=2pt] (-1.763,-2.427) -- (-1.176,-1.618);
		\draw[line width=2pt] (-2.853,0.927) -- (-1.902,0.618);
		\draw[line width=2pt] (0,2)--(1.176,-1.618);
		\draw[line width=2pt] (0,2)--(-1.176,-1.618);
		\draw[line width=2pt] (1.176,-1.618)--(-1.902,0.618);
		\draw[line width=2pt] (1.902,0.618)--(-1.902,0.618);
		\draw[line width=2pt] (-1.176,-1.618)--(1.902,0.618);
		\filldraw (0,2) circle (0.1cm);
		\filldraw (1.902,0.618) circle (0.1cm);
		\filldraw (-1.902,0.618) circle (0.1cm);
		\filldraw (1.176,-1.618) circle (0.1cm);
		\filldraw (-1.176,-1.618) circle (0.1cm);
		\filldraw (0,3) circle (0.1cm);
		\filldraw (2.853,0.927) circle (0.1cm);
		\filldraw (1.763,-2.427) circle (0.1cm);
		\filldraw (-2.853,0.927) circle (0.1cm);
		\filldraw (-1.763,-2.427) circle (0.1cm);
		\draw (0.75,3) node {$y_3$};
		\draw (0.5,2) node {$x_1$};
		\draw (1.676,-1.618) node {$x_3$};
		\draw (-1.676,-1.618) node {$x_4$};
		\draw (2.302,0.218) node {$x_2$};
		\draw (-2.302,0.218) node {$x_5$};
		\draw (3.253,0.527) node {$y_4$};
		\draw (-3.253,0.527) node {$y_2$};
		\draw (2.5,-2.427) node {$y_5$};
		\draw (-2.5,-2.427) node {$y_1$};
	\end{tikzpicture}
	\end{center}
	\caption{The Legal Complex $\Delta_{\weight{2}{3},C_5}$}
	\label{fig:cyclesc}
\end{figure}
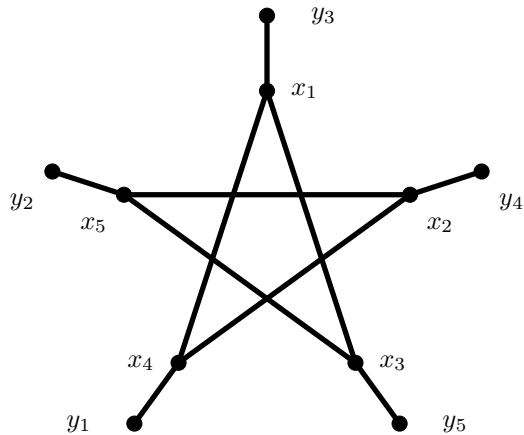

By Propositions \ref{thm:cycle1} and \ref{thm:cycle2} we have
\begin{align*}
	f_0&=1\\
	f_1&=2n=10,\\
	f_2&=\frac{n(n-2a+1)}{2}+n(n-a-b+1)=10,
\end{align*}
and since $\max\{\left\lfloor n/a\right\rfloor, \left\lfloor n/b\right\rfloor\}=2$, we get the $f$-vector $(1, 10, 10)$, which can be verified from the simplicial complex.

We will compare these with the Kruskal-Katona bound. The $i$-canonical representations and $j$th pseudopowers are:
\begin{center}
\begin{tabular}{ll}
	$f_1=\binom{10}{1}$ & $f_1^{(2)}=\binom{10}{2}=45$\\
	$f_2=\binom{5}{2}$ & $f_2^{(3)}=\binom{5}{3}=10$\\
	 & $f_2^{(1)}=\binom{5}{1}=5$
\end{tabular}
\end{center}
Then $f_2=10<f_1^{(2)}=45$, $f_3=0<f_2^{(3)}=10$, and $f_1=10>f_2^{(1)}=5$, showing that for this example Propositions \ref{thm:cycle1} and \ref{thm:cycle2} give improved necessary conditions for a vector to be the $f$-vector of a legal complex of a placement game played on a cycle over the ones given in the Kruskal-Katona theorem.
\end{example}

Similar to placement games on a path, we have that for fixed $a$ and $b$ and sufficiently large $n$ the bound in Proposition \ref{thm:cycle2} on $f_2$ is better than the Kruskal-Katona bound. By the Kruskal-Katona theorem we have
\begin{align*}
	f_2\le f_1^{(2)}&=\binom{2n}{2}\\
		&=\frac{1}{2}\left[4n^2+n(-2)\right],
\end{align*}
whereas Proposition \ref{thm:cycle2} gives
\begin{align*}
	f_2= &\frac{n(n-2a+1)}{2}+\frac{n(n-2b+1)}{2}+n(n-a-b+1)\\
	=&\frac{1}{2}\left[4n^2+n(4-4a-4b)\right]\\
	<&\frac{1}{2}\left[4n^2+n(-2)\right],
\end{align*}
since $a,b\ge 1$ implies $4-4a-4b\le -4$, showing that as $n$ grows larger our bound becomes increasingly better than the Kruskal-Katona bound.

\section{Playing on the Complete Graph $K_n$}
Finally, we will consider placement games played on a complete graph of $n$ vertices in which Left places pieces of weight $a$ and Right pieces of weight $b$.
\begin{proposition}\label{thm:complete}
If a simplicial complex is the legal complex of $\weight{a}{b}$ played on $K_n$ then
\begin{equation}\label{eq:completeab}
	f_k=\sum_{l=0}^k\left(\frac{\displaystyle\prod_{i=0}^{k-l-1}\binom{n-ia}{a}}{(k-l)!}\right)\left(\frac{\displaystyle\prod_{j=0}^{l-1}\binom{n-(k-l)a-jb}{b}}{l!}\right)
\end{equation}
for $k\ge 0$.
\end{proposition}
\begin{proof}
Playing a piece of weight $a$ on the complete graph with $n$ vertices is equivalent to deleting $a$ vertices from the graph. Thus placing a second piece on the graph is equivalent to placing a piece on the complete graph on $n-a$ vertices.

Also, since every pair of vertices is connected, playing a piece of weight $a$ is equivalent to playing $a$ pieces of weight $1$, thus there are $\binom{n}{a}$ choices for placing the piece.

Thus playing $s$ pieces of weight $a$ we have
\[\frac{\prod_{i=0}^{s-1}\binom{n-ia}{a}}{s!}\]
choices. Then playing $k-l$ pieces of weight $a$ and $l$ pieces of weight $b$ (assuming without loss of generality we place the pieces of weight $a$ first) we have
\[\frac{\prod_{i=0}^{k-l-1}\binom{n-ia}{a}}{(k-l)!}\;\frac{\prod_{j=0}^{l-1}\binom{n-(k-l)a-jb}{b}}{l!}\]
different positions.

To get the total number of positions with $k$ pieces played, we let $l$ range from $0$ to $k$ and add the terms, giving Equation \ref{eq:completeab}.
\end{proof}

If $a=b$, then the previous bound becomes
\begin{align*}
	f_k&=\sum_{l=0}^k\frac{n(n-1)\cdots(n-(k-l)a+1)(n-(k-l)a)\cdots(n-ka+1)}{(k-l)!l!(a!)^k}\\
		&=\frac{n!}{(n-ka)!(a!)^k}\sum_{l=0}^k\frac{1}{k!}\binom{k}{l}\\
		&=\frac{n!}{(n-ka)!k!(a!)^k}\sum_{l=0}^k\binom{k}{l}\\
		&=\frac{n!}{(n-ka)!k!(a!)^k}2^k.
\end{align*}

If $a=b=1$, then this becomes
\begin{align*}
	f_k&=\frac{n!}{(n-k)!k!}2^k\\
		&=\binom{n}{k}2^k
\end{align*}
which is the bound given in Proposition \ref{thm:weight1bound}.

If we assume without loss of generality that $a\le b$, then we have
\begin{align*}
	f_k&=\sum_{l=0}^k\frac{n(n-1)\cdots(n-(k-l)a-lb+1)}{(k-l)!l!(a!)^{k-l}(b!)^l}\\
		&=\frac{n!}{k!}\sum_{l=0}^k\frac{\binom{k}{l}}{(a!)^{k-l}(b!)^l(n-(k-l)a-lb)!}\\
		&\le\frac{n!}{k!}\sum_{l=0}^k\frac{\binom{k}{l}}{(a!)^k(n-kb)!}\\
		&=\frac{n!}{(n-kb)!k!(a!)^k}2^k.
\end{align*}

We can similarly find a lower bound. Thus
\[\frac{n!}{(n-ka)!k!(b!)^k}2^k\le f_k\le \frac{n!}{(n-kb)!k!(a!)^k}2^k.\]
For fixed $a, b,$ and $k$, we then have
\[
	n(n-1)\cdots(n-ka+1)\frac{2^k}{k!(b!)^k}\le f_k\le n(n-1)\cdots(n-kb+1)\frac{2^k}{k!(a!)^k},
\]
and since \[n(n-1)\cdots(n-ka+1)\ge (n-ka+1)^{ka} \text{ and } n(n-1)\cdots(n-kb+1)\le n^{kb},\] this implies
\[
	C'(n-ka+1)^{ka}=C'n^{ka}+O(n^{ka-1})\le f_k\le Cn^{kb},
\]
where $C$ and $C'$ are constants depending on $a$ and $k$, respectively $b$ and $k$.

Also note that $\weight{a}{b}$ played on the complete graph $K_n$ is the least restrictive game on the most connected board. Thus the formula in Proposition \ref{thm:complete} gives upper bounds for any placement game with weights on any board.

\begin{example}\label{ex:complete}
Consider $\weight{2}{2}$ and let the board be the complete graph $K_4$. Let $x_{i,j}$ represent a Left piece occupying the vertices $i$ and $j$, and similarly for $y_{i,j}$. For example the position in Figure \ref{fig:completeex} is represented by $x_{1,4}y_{2,3}$.

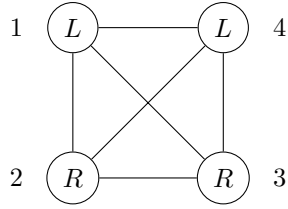
\begin{figure}[!ht]
	\begin{center}
	\begin{tikzpicture}
		\node[shape=circle, draw] (1) at (0,0) {$L$};
		\node[shape=circle, draw] (2) at (2,0) {$L$};
		\node[shape=circle, draw] (3) at (0,-2) {$R$};
		\node[shape=circle, draw] (4) at (2,-2) {$R$};
		\draw (1) to (2) to (3) to (4) to (1) to (3);
		\draw (2) to (4);
		\node at (-0.75,0) {$1$};
		\node at (-0.75,-2) {$2$};
		\node at (2.75,0) {$4$};
		\node at (2.75,-2) {$3$};
	\end{tikzpicture}
	\end{center}
	\caption{An Example Position for $\weight{2}{2}$ on $K_4$}
	\label{fig:completeex}
\end{figure}

The corresponding simplicial complex is given in Figure \ref{fig:completesc}.
\begin{figure}[!ht]
	\begin{center}
	\begin{tikzpicture}[scale=1.5]
		\draw[line width=1.3] (0,0) -- (0,1) -- (1,1) -- (1,0)--cycle;
		\draw[line width=1.3] (2,0) -- (3,0) -- (3,1) -- (2,1) -- cycle;
		\draw[line width=1.3] (4,0) -- (5,0) -- (5,1) -- (4,1) -- cycle;
		\filldraw (0,0) circle (0.066cm);
		\filldraw (0,1) circle (0.066cm);
		\filldraw (1,0) circle (0.066cm);
		\filldraw (1,1) circle (0.066cm);
		\filldraw (2,0) circle (0.066cm);
		\filldraw (2,1) circle (0.066cm);
		\filldraw (3,0) circle (0.066cm);
		\filldraw (3,1) circle (0.066cm);
		\filldraw (4,0) circle (0.066cm);
		\filldraw (4,1) circle (0.066cm);
		\filldraw (5,0) circle (0.066cm);
		\filldraw (5,1) circle (0.066cm);
		\draw (0,1.3) node {$x_{1,2}$};
		\draw (1,1.3) node {$x_{3,4}$};
		\draw (2,1.3) node {$x_{1,3}$};
		\draw (3,1.3) node {$x_{2,4}$};
		\draw (4,1.3) node {$x_{1,4}$};
		\draw (5,1.3) node {$x_{2,3}$};
		\draw (1,-0.4) node {$y_{1,2}$};
		\draw (0,-0.4) node {$y_{3,4}$};
		\draw (3,-0.4) node {$y_{1,3}$};
		\draw (2,-0.4) node {$y_{2,4}$};
		\draw (5,-0.4) node {$y_{1,4}$};
		\draw (4,-0.4) node {$y_{2,3}$};
	\end{tikzpicture}
	\end{center}
	\caption{The Legal Complex $\Delta_{\weight{2}{2},K_4}$}
	\label{fig:completesc}
\end{figure}
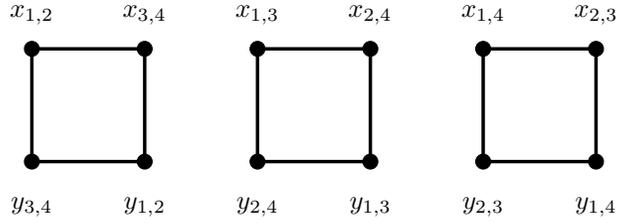

By Proposition \ref{thm:complete} we have
\begin{align*}
	f_0&=1\\
	f_1&=\binom{n}{a}+\binom{n}{b}=12,\\
	f_2&=\frac{\binom{n}{a}\binom{n-a}{a}}{2}+\binom{n}{a}\binom{n-a}{b}+\frac{\binom{n}{b}\binom{n-b}{b}}{2}=12,
\end{align*}
and since $\max\{\left\lfloor n/a\right\rfloor, \left\lfloor n/b\right\rfloor\}=2$, we get the $f$-vector $(1, 12, 12)$, which can be verified from the simplicial complex.

The $i$-canonical representations and the $j$th pseudopowers are:
\begin{center}
\begin{tabular}{ll}
	$f_1=\binom{12}{1}$ & $f_1^{(2)}=\binom{12}{2}=66$\\
	$f_2=\binom{5}{2}+\binom{2}{1}$ & $f_2^{(3)}=\binom{5}{3}+\binom{2}{2}=11$\\
	 & $f_2^{(1)}=\binom{5}{1}+\binom{2}{0}=6$
\end{tabular}
\end{center}
Then $f_2=12<f_1^{(2)}=66$, $f_3=0<f_2^{(3)}=11$, and $f_1=12>f_2^{(1)}=6$, showing that for this example the formula in Proposition \ref{thm:complete} gives improved necessary conditions for a vector to be the $f$-vector of a legal complex.
\end{example}

We will now show that for fixed $a$ and $b$ and sufficiently large $n$, the bound in Proposition \ref{thm:complete} for $f_2$ is better than the Kruskal-Katona bound. By the Kruskal-Katona theorem we have
\begin{align*}
	f_2\le f_1^{(2)}&=\binom{\binom{n}{a}+\binom{n}{b}}{2}\\
		&=\frac{1}{2}\left[\binom{n}{a}\left(\binom{n}{a}+2\binom{n}{b}-1\right)+\binom{n}{b}\left(\binom{n}{b}-1\right)\right],
\end{align*}
whereas Proposition \ref{thm:complete} gives
\begin{align*}
	f_2= &\frac{1}{2}\binom{n}{a}\binom{n-a}{a}+\frac{1}{2}\binom{n}{b}\binom{n-b}{b}+\binom{n}{a}\binom{n-a}{b}\\
	=&\frac{1}{2}\left[\binom{n}{a}\left(\binom{n-a}{a}+2\binom{n-a}{b}\right)+\binom{n}{b}\binom{n-b}{b}\right].
\end{align*}

Recall that $f(n)= O(g(n))$ means that $f(n)\le Cg(n)$ for some positive constant $C$. Then $f(n)= O(n^k)$ means that $f(n)$ is bounded by a polynomial of degree at most $k$. Also recall that $f(n)=g(n)+O(n^k)$ means $f(n)-g(n)= O(n^k)$.

Since 
\begin{align*}
	\binom{n}{i}&=\frac{1}{i!}\left(n^i-n^{i-1}\frac{i(i-1)}{2}+O(n^{i-2})\right)\text{ for } i\ge 2\\
	\binom{n-i}{j}&=\frac{1}{j!}\left(n^j-n^{j-1}\frac{j(j+2i-1)}{2}+O(n^{j-2})\right)\text{ for } j\ge 2\\
\end{align*}
it easily follows that $\binom{n-a}{a}+2\binom{n-a}{b}\le\binom{n}{a}+2\binom{n}{b}-1$ and $\binom{n-b}{b}\le\binom{n}{b}-1$. Thus
\begin{align*}
&\frac{1}{2}\left[\binom{n}{a}\left(\binom{n-a}{a}+2\binom{n-a}{b}\right)+\binom{n}{b}\binom{n-b}{b}\right]\\
&< \frac{1}{2}\left[\binom{n}{a}\left(\binom{n}{a}+2\binom{n}{b}-1\right)+\binom{n}{b}\left(\binom{n}{b}-1\right)\right],
\end{align*}
showing that the new bound is better than the Kruskal-Katona bound as $n$ grows larger.

We have not compared the bounds for $f_k$ with $k>2$ since it is difficult to find the $i$-canonical representation of $f_{k-1}$ in this case.

\section{Discussion}

A general question is to find sufficient conditions for a simplicial complex to be a legal complex. Since it is already not easy to find necessary conditions for a vector to be the $f$-vector of a legal complex, this seems to be very hard and much further work is needed.

\end{document}